\documentclass{amsart}
\usepackage{amsmath, amsthm,amssymb,amscd,amsxtra, booktabs}
\usepackage[all,cmtip]{xy}
\usepackage{mathabx}
\usepackage{graphicx}
\usepackage[italian, english]{babel}
\usepackage[applemac]{inputenc} 
\DeclareMathOperator{\Ext}{Ext}

\newcommand{\field}[1]{\mathbb{#1}}

\newcommand{\Z}{\field{Z}}
\newcommand{\e}{\epsilon}
\newcommand{\F}{\field{F}}

\newcommand{\N}{\field{N}}

\newcommand{\ep}{{ \varepsilon }}
\newtheorem {theorem}{Theorem}[section]

\newtheorem {corollary}[theorem]{Corollary}
\theoremstyle{definition}
\newtheorem {definition}[theorem]{Definition}

\theoremstyle{remark}
\numberwithin{equation}{section}
\begin{document}
\bibliographystyle{plain}
\title[Operations for Hopf Cohomology]{A Note on the Algebra of Operations for Hopf Cohomology at Odd Primes}
\author{Maurizio Brunetti}\author{Adriana Ciampella}\author{Luciano A. Lomonaco}
\keywords{Steenrod Algebra, Invariant Theory}
\address{\, \newline Dipartimento di Matematica e applicazioni,\newline Universit\`a di Napoli Federico II, \newline Piazzale Tecchio 80   I-80125 Napoli, Italy. 
\newline \newline E-mail: mbrunett@unina.it,  ciampell@unina.it, lomonaco@unina.it}
\begin{abstract} 
Let $p$ be any prime, and let  ${\mathcal B}(p)$ be the algebra of operations on the cohomology ring of any cocommutative $\F_p$-Hopf algebra. In this paper we show that when $p$ is odd (and unlike the $p=2$ case), ${\mathcal B}(p)$  cannot become an  object in the Singer category of  $\F_p$-algebras with coproducts, if we require that coproducts act on the generators of ${\mathcal B}(p)$ coherently with their nature of cohomology operations. \end{abstract}
\subjclass[2010]{55S10, 55T15}
\maketitle
\section{Introduction}
After noticing that the algebra $\mathcal B (p)$  of Steenrod operations on  $\Ext_{\Lambda}(\Z_{p}, \Z_{p})$, the cohomology of a
graded cocommutative Hopf algebra $\Lambda$ over $\F_{p}$,
 is (not even only for $p=2$) neither a Hopf algebra nor a bialgebra,   William B. Singer introduced in \cite{S} the notions, one dual to the other, of a $k$-{\em algebra with coproducts} and $k$-{\em coalgebra with products}, for any commutative ring $k$,  arguing that this is the right categorial setting to study $\mathcal B(2)$ and its dual.  
Further examples of $k$-algebras with coproducts appeared in literature in the last decade. For instance the third author studied in \cite{L3} those arising as invariants of finitely generated $\F_2$-polynomial algebras under the action of the general linear groups and their upper triangular subgroups.  

 More recently \cite{BCL1_5}, the authors have taken into account ${\mathcal B}(p)$, when $p$ is an odd prime. Such algebra  has been described in terms of generators and relations by Liulevicius in  \cite{Li}.
The important role of $\mathcal B(p)$ for stable homotopy computations is well and long established (see for example \cite{Br1}, \cite{Br2}, \cite{Br3}, \cite{Ka}, \cite{Li}). Furthermore a relevant subalgebra of $\mathcal B (p)$ is a quotient of the universal Steenrod algebra $\mathcal Q(p)$, introduced in \cite{kn:M} and broadly examined by the authors (\cite{BC}-\cite {CL2}, \cite{L1},\cite{L2}, \cite{L4}).
 Along the spirit of \cite{S}, in \cite{BCL1_5} the authors equipped  $\mathcal B(p)$ with a suitable collection of $\F_p$-linear mappings that made it the underlying set of an object in the Singer category of $\F_p$-algebras with coproducts.
 Yet, in \cite{BCL1_5}, the chosen coproduct acting on the Bockstein operator $\beta$ has little to do with its nature of cohomology operation. 
 
Our Theorem \ref{t2} states that $\mathcal B (p)$ does not admit a structure of $\F_{p}$-algebra with coproducts  consistent with \eqref{HH}, and hence with the geometric meaning of all its generators. 

A comparison between Theorems \ref{t1} and \ref{t2} shows that the non-primitivity of the Bockstein operator stands as unavoidable obstruction. This is an interesting phenomenon that deserves to be further investigated. In fact it suggests that Singer's notion of {\em algebra with coproducts} in \cite{S} needs to be refined, or perhaps that the algebra $\mathcal B (p)$  is a deformation of a certain geometrically significant (and yet-to-be-determined) algebraic object.

 
 \section{A Theorem of Non-existence}
 Let $p$ be an odd prime. We recall that the algebra $\mathcal B (p)$ of Steenrod operations on the cohomology ring of any cocommutative Hopf $\F_p$-algebra $\Lambda$ is generated by
\begin{equation}\label{b1} 
P^k:\Ext_{\Lambda}^{q,t}(\Z_{p}, \Z_{p})\to \Ext_{\Lambda}^{q+2k(p-1), pt}(\Z_{p}, \Z_{p}) \quad (k,q,t \geq 0), \end{equation}
and
\begin{equation}\label{b2}
\beta: \Ext_{\Lambda}^{q,t}(\Z_{p}, \Z_{p})\to \Ext_{\Lambda}^{q+1, pt}(\Z_{p}, \Z_{p}) \quad \quad (q,t \geq 0) 
\end{equation}
subject to the following relations (see \cite{Li}):
\begin{equation}\label{ar1}    
\beta^2 =0, 
\end{equation}
\begin{equation}\label{ar2}  
 P^aP^b =\sum_{t=0}^{\lfloor \frac{a}{p} \rfloor}A(b,a,t)\; P^{a+b-t}P^t \quad \mbox{when $a<pb$,} 
 \end{equation}
\begin{equation}\label{ar3} 
P^a \beta P^b=\sum_{t=0}^{\lfloor \frac{a}{p} \rfloor}B(b,a,t)\; \beta P^{a+b-t}P^t +\sum_{t=0}^{\lfloor \frac{a-1}{p} \rfloor}   A(b,a-1,t)\; P^{a+b-t}\beta P^t  \quad \mbox{when $a\leq pb$.}
\end{equation}
  Coefficients in the several sums of \eqref{ar2} and \eqref{ar3} read as follows:
  $$A(k,r,j)= (-1)^{r+j}{(p-1)(k-j)-1 \choose r-pj} $$ 
and $$ B(k,r,j) = (-1)^{r+j}{(p-1)(k-j) \choose r-pj}.$$ 

Unlike the element with the same name in the ordinary Steenrod algebra $\mathcal A_p$, $P^{0}$ in $\mathcal B (p)$ is not the identity. 
    
 We now recall the definition of {\em $k$-algebra with coproducts}.
 
\begin{definition}\label{def} A {\em $k$-algebra with coproducts} is a bigraded unital algebra $$\mathcal C= \{ \mathcal C_{n,s} \, | \, n,s \geq 0 \} $$ together with degree preserving maps $\e_s: \mathcal C_{*,s} \rightarrow k$ and $\psi_s : \mathcal C_{*,s} \rightarrow \mathcal C_{*,s} \otimes \mathcal C_{*,s}$ for each $s \geq 0$ such that:
\begin{enumerate} 
\item $\mathcal C_{*,s}$ is a graded coalgebra  with counit $\e_s$ and coproduct $\psi_s$, for each $s \geq 0$;
\item the algebra unit $\eta: k \rightarrow \mathcal C_{*,0}$ is a map of coalgebras;
 \item  the multiplication $\mu : \oplus_{h+k=s} (\mathcal C_{*,h} \otimes \mathcal C_{*,k}) \rightarrow \mathcal C_{*,s}$ preserves the coalgebra structure for each $s \geq 0$. 
 \end{enumerate} 
  \end{definition}
 As already noted in  \cite{BCL1_5}, Item (iii) of Definition \ref{def} makes sense: the category of graded coalgebras has tensor products and sums,
 and the category of graded algebras has tensor products and categorical products. 
Explicitly, given two graded algebras $A$ and $B$, on $A \otimes B$ we assume defined  the product
\begin{equation}\label{yyy} (a \otimes b) (c \otimes d) = (-1)^{(\deg b)(\deg c)} (ac \otimes bd). \end{equation} 
It follows in particular that $\mathcal D_{h,k} = \mathcal C_{*,h} \otimes \mathcal C_{*,k}$ is a coalgebra, and a comultiplication on  $\mathcal E_s = \oplus_{h+k=s} \mathcal D_{h,k} $ defined coordinatewise makes  $\mathcal E_s$ itself a coalgebra.

Note also that Item (iii) essentially says that each map in the family of maps $\{\psi_s \}$ is completely determined by `extending multiplicatively' the action on the elements of a generating set of $\mathcal C$.


In  \cite{BCL1_5}, we proved the following Theorem.
\begin{theorem}\label{t1} Let $p$ be an odd prime. Once assigned the bidegree 
\begin{equation}\label{bd} |P^s|=(2s(p-1), 1), \quad \quad |\beta|=(1,0), \end{equation}
to its algebra generators,  $\mathcal B(p)$ admits a unique structure as a $\F_{p}$-algebra with coproducts, where 
 \begin{equation}\label{y} \psi_{0}(\beta)  =\beta\otimes 1 + 1 \otimes \beta  \quad \mbox{and} \quad \psi_{1}(P^s)=\sum_{i+j=s}P^i\otimes P^j\quad \quad (s\geq 0). \end{equation}    \end{theorem}  
 The unique structure the statement referred to turned out to be the dual  of  a suitable  $\F_p$-{\em coalgebra with products} in the sense of \cite{S}. 
 
 The proof uses the fact that the  $\F_p$-vector space $\mathcal B(p)_{*,s}$ has a basis made by admissible monomials, i.e. monomials of type 
\begin{equation}\label{g} \beta^{\ep_0}\,P^{t_1}\,\beta^{\ep_1}\,P^{t_2} \, \cdots \, \beta^{\ep_{s-1}}\, P^{t_s}  \, \beta^{\ep_s} \end{equation}
where $\ep_i \in \{0,1 \}$, and $ t_j \geq pt_{j+1}+ \ep_j$  for $ 1 \leq j <s$ (see Proposition 3.14 in \cite{BCL1_5}).

At a careful examination, Theorem \ref{t1} cannot be viewed  as the odd  $p$-counterpart of Theorem 1.2 in \cite{S}.

 In fact the Bockstein operator $\beta$ acts on products in cohomology rings of Hopf algebras according to the formula
 $$ \beta (uv) = \beta (u) \cdot P^0 (v) + (-1)^{\vert u \rvert} P^0(u) \cdot \beta (v)$$
 (see Equation 3.2.5 in \cite{Li}). Consequently, a coproduct $\tilde{\psi}$ that would take such behaviour into account should satisfy
\begin{equation}\label{HH}  \tilde{\psi} (\beta) = \beta \otimes P^0 + P^0 \otimes \beta. \end{equation}
It is quite natural to ask whether $\mathcal B (p)$, after suitably regrading its generators, admits a structure of $\F_{p}$-algebra with coproducts  consistent with \eqref{HH}. 
Theorem \ref{t2} answers negatively to such question.
\begin{theorem}\label{t2} Let $p$ be an odd prime. There is no way to doubly filter $\mathcal B (p)$ in order to make it a $\F_{p}$-algebra with coproducts, if we require that coproducts are defined consistently with 
  \begin{equation}\label{ytz}  \beta \longmapsto \beta \otimes P^0 + P^0 \otimes \beta
  \end{equation}
  and
    \begin{equation}\label{ytz2} 
 P^s \longmapsto \sum_{i+j=s}P^i\otimes P^j\quad \quad (s\geq 0). \end{equation}  
  \end{theorem}  
  \begin{proof}
  We argue by contradiction. Suppose there exists an $\F_p$-algebra with coproducts 
  \begin{equation}\label{a1}
\Big( \tilde{\mathcal B} = \{ \tilde{\mathcal B}_{n,s} \, | \, n,s \geq 0 \}, \, \mu, \, \eta, \{ \e_s \, \vert \, s \geq 0 \},  \, \{ \tilde{\psi}_s \, \vert \, s \geq 0 \}  \Big), \end{equation}
where
$$   \bigcup \tilde{\mathcal B}_{n,s} = \mathcal B (p),$$ 
and the coproducts in $\{ \tilde{\psi}_s \, \vert \, s \geq 0 \} $ are consistent with \eqref{ytz} and  \eqref{ytz2}.

Definition \ref{def} in particular implies that $\psi_s ( \mathcal C_{*,s}) \subseteq \mathcal C_{*,s} \otimes \mathcal C_{*,s}$. By \eqref{ytz} and \eqref{ytz2} it follows that the Bockstein operator $\beta $ and the Steenrod powers $P^i$ ($i \geq 0)$ all belong to the same coalgebra $\mathcal C_{*,\bar{s}}$ for a suitable $\bar{s} \in \N_0$. 
Let $\bar{r}$ and $\bar{t}$ be the non-negative integers such that 
$$ \beta \in \mathcal C_{\bar{r},\bar{s}} \quad \text{and} \quad P^0 \in \mathcal C_{\bar{t},\bar{s}}.$$
By Item (iii) of Definition \ref{def} we get
 \begin{align*}
\tilde{\psi}_{2\bar{s}} (\beta^2) &=\tilde{\psi}_{\bar{s}} (\beta) \tilde{\psi}_{\bar{s}} (\beta)\\
       &=(\beta \otimes P^0 + P^0 \otimes \beta)(\beta \otimes P^0 + P^0 \otimes \beta).
\end{align*}
The latter equality comes from \eqref{ytz}. Recalling \eqref{yyy}  and the fact that $\beta^2=0$ by \eqref{ar1},  we obtain
\begin{equation}\label{fin} \tilde{\psi}_{2\bar{s}} (\beta^2)=\tilde{\psi}_{2\bar{s}} (0) = (-1)^{\bar{t}^2}\beta P^0 \otimes P^0 \beta + (-1)^{\bar{r}^2} P^0 \beta \otimes \beta P^0.
\end{equation}
Equation \eqref{fin} contradicts the $\F_p$-linearity of  $\tilde{\psi}_{2\bar{s}}$, in fact 
$\beta P^0$  and $P^0\beta$ are both non-zero in $\mathcal B(p)$ and non-proportional.\end{proof}

\section{Toward Further Investigation}
Theorem~\ref{t2} does not foil the attempt to provide proper subalgebras of $\mathcal B(p)$ with a structure  of an $\F_{p}$-algebra with coproducts, as the next Proposition shows.
\begin{corollary} Let $\mathcal P(p)$ be the subalgebra of $\mathcal B(p)$ generated by the set $\{ P^i \, \vert \, i \geq 0 \}$ of pure powers. Once assigned the bidegree 
\begin{equation}\label{bz} |P^s|=(2s(p-1), 1), \end{equation}
the algebra $\mathcal P(p)$ admits a unique structure as a $\F_{p}$-algebra with coproducts, where 
 \begin{equation}\label{zzz} \psi_{1}(P^s)=\sum_{i+j=s}P^i\otimes P^j\quad \quad (s\geq 0). \end{equation}   
\end{corollary}
\begin{proof} Once you input Theorem~\ref{t1}, the only relevant point is the absence of $\beta$ among the generating relations \eqref{ar2} in $\mathcal P(p)$.

\end{proof}
In \cite{BCL1_5}, the authors took into account the subalgebra $\mathcal C(p)$ of $\mathcal B(p)$ generated by the set $\{ P^i, \, \beta P^i \, \vert \, i \geq 0 \}$. There is a good reason to believe that 
 $\mathcal C(p)$ could be made an object in Singer's category with coproducts consistent with \eqref{ytz} and \eqref{ytz2}. In fact, assigned the bidegree 
 \begin{equation}\label{bjj} |P^s|=(2s(p-1), 1), \quad \quad |\beta P^s|=(2s(p-1)+1,2), \end{equation} 
 and set
 $$ \psi_2 (\beta P^0) = \beta P^0 \otimes P^0P^0 + P^0P^0 \otimes  \beta P^0,$$
 the map $\psi_4$ could be $\F_p$-linear since the element $\beta P^0\beta P^0=0$ would be mapped onto
 $( \beta P^0 \otimes P^0P^0 + P^0P^0 \otimes  \beta P^0)^2$ which can be proved to vanish by \eqref{ar3} and \eqref{yyy}. The proof that all relations are preserved will depend on the existence of an appropriate $\F_p$-coalgebra with products: the dual to the required structure on $\mathcal C(p)$.
  
\end{document}